\documentclass[11pt]{article}%
\usepackage{amsfonts}
\usepackage{amssymb}
\usepackage{amsmath}
\usepackage[margin=2cm]{geometry}
\usepackage{graphicx}%
\providecommand{\U}[1]{\protect\rule{.1in}{.1in}}
\newtheorem{theorem}{Theorem}

\newtheorem{lemma}[theorem]{Lemma}

\newenvironment{proof}[1][Proof]{\noindent\textbf{#1.} }{\ \rule{0.5em}{0.5em}}
\title{The infinitely many zeros of stochastic coupled oscillators driven by random forces}
\author{H. de la Cruz  \\
	EMAp-FGV / hugo.delacruz@fgv.br  \\
	\and
	J.C.Jimenez \\
	ICIMAF / jcarlos@icimaf.cu \\
\and
	R.J.Biscay\\
	CIMAT / rolando.biscay@cimat.mx \\
	}

\date{ }

\begin{document}

\maketitle

\begin{abstract}
In this work, previous results concerning the infinitely many zeros of
single stochastic oscillators driven by random forces are extended to the general class
of coupled stochastic oscillators. We focus on three main subjects: 1) the analysis of
this oscillatory behavior for the case of coupled harmonic oscillators; 2) the
identification of some classes of coupled nonlinear oscillators showing
this oscillatory dynamics and 3) the capability of some numerical
integrators - thought as discrete dynamical systems - for reproducing the
infinitely many zeros of coupled harmonic oscillators driven by random forces.
\end{abstract}

\section{Introduction}

Motivated by their capability to describe the time evolution of complex random
phenomena, models of nonlinear oscillators driven by random forces have become
a focus of intensive studies (see, e.g., \cite{Gitterman05},
\cite{Anishchenko13}, \cite{Stratonovich67}, \cite{Kliemann95}, \cite{Mao07}).
Naturally, the added noise modifies the dynamics of the deterministic
oscillators and so new distinctive dynamical features arise in these random
systems. Since the complexity of the random dynamics depends on the type of
nonlinearity and the level of noise, many of the results on this matter have
been achieved for specific classes of stochastic oscillators. In particular, a
number of properties have been studied for the simple harmonic oscillator such
as the stationary probability distribution, the linear growth of energy along
the paths, the oscillation of the solution, and the symplectic structure of
Hamiltonian oscillators, among others (see, e.g., \cite{Markus88},
\cite{Bismut81}, \cite{Markus93}, \cite{Schurz 09}, \cite{Schurz08}). Some of
these properties have been also analyzed for coupled harmonic oscillators.
However, to the best of our knowledge, there are no studies concerning the
oscillatory behavior around the origin of stochastic coupled oscillators.

On the other hand, demanded by an increasing number of practical applications
(see e.g., \cite{Ghorbanian15}, \cite{Anishchenko13}, \cite{Ghorbanian15b},
and references therein), the numerical simulation of stochastic oscillators
has also a high interest. In particular, it is required to use specialized
numerical integrators that preserve the dynamics of the oscillators since
general multipurpose integrators fail to achieve this target. This is so
because, in general, the dynamics of discrete dynamical systems is far richer
than that of the continuous ones. Consequently, specific oriented integrators
for stochastic oscillators have also been proposed, for instance, in
\cite{delaCruz16}, \cite{Milstein02}, \cite{Cohen12}, \cite{Tocino07},
\cite{Senosiain14}. Distinctively, in \cite{delaCruz16}, the family of the
Locally Linearized methods have been proved to simultaneously reproduce
various dynamical properties of the stochastic harmonic oscillators including
the oscillatory behavior around $0$ of the single oscillators.

In this work, we are interested in the study of the oscillatory behavior of
the stochastic coupled oscillators driven by random forces. We focus on three
main aspects: 1) the analysis of this oscillatory behavior for the case of
coupled harmonic oscillators, a property that has only been demonstrated for
simple oscillators (\cite{Markus88},\cite{Mao07}); 2) the identification of
some classes of coupled nonlinear oscillators that display this dynamics; and
3) the capability of the Locally Linearized integrators - as discrete
dynamical systems - of reproducing the infinitely many zeros of the coupled
harmonic oscillators driven by random forces, which complements known results
of these integrators for simple harmonic oscillators \cite{delaCruz16}%
.\newline

\section{The infinitely many zeros of the coupled harmonic oscillators}

Let us first consider the undamped harmonic oscillator, defined by the
$2d$-dimensional Stochastic Differential Equation (SDE) with additive noise
\begin{equation}
d\mathbf{x}\left(  t\right)  =\mathbf{Ax}\left(  t\right)  dt+\mathbf{B}%
d\mathbf{w}_{t}, \label{harmonic oscillator}%
\end{equation}
for $t\geq t_{0}\geq0$, with initial condition $\mathbf{x}(t_{0})=(x_{0}%
,y_{0})^{\top},$ $x_{0},y_{0}\in%
\mathbb{R}
^{d}$ and $d>1$. Here,
\[
\mathbf{x}(t)=%
\begin{bmatrix}
x(t)\\
y(t)
\end{bmatrix}
,\text{ \ \ \ \ \ \ \ \ }\mathbf{A}=%
\begin{bmatrix}
\mathbf{0} & \mathbf{I}\\
-\Lambda^{2} & \mathbf{0}%
\end{bmatrix}
,\text{ \ \ \ \ and\qquad}\mathbf{B}=%
\begin{bmatrix}
\mathbf{0}\\
\Pi
\end{bmatrix}
,
\]
being $\Lambda\in%
\mathbb{R}
^{d\times d}$ a nonsingular symmetric matrix, $\Pi\in%
\mathbb{R}
^{d\times m}$ a matrix, $\mathbf{I}$ the $d-$dimensional identity matrix, and
$\mathbf{w}_{t}$ an $m$-dimensional standard Wiener process on the filtered
complete probability space $\left(  \Omega,\mathfrak{F,}\left(  \mathfrak{F}%
_{t}\right)  _{t\geq t_{0}},\mathbb{P}\right)  $.

In what follows, the symbol $\left\langle \cdot,\cdot\right\rangle $ denotes
the Euclidean scalar product associated to the Euclidean vector norm
$\left\vert \cdot\right\vert $. For matrices, $\left\vert \cdot\right\vert $
denotes the Frobenious matrix norm. In addition, the following lemma will be useful.

\begin{lemma}
\label{Lemma1}Let $\eta_{1},...,\eta_{n},...$ be independent $\mathcal{N}%
(0,1)$ distributed random variables. Let $\{\sigma_{nr}\}$ be a bounded
triangular array of real numbers. Set $S_{n}=$ $%
{\textstyle\sum\limits_{r=1}^{n}}
\sigma_{nr}\eta_{r}$ and $s_{n}^{2}=%
{\textstyle\sum\limits_{r=1}^{n}}
\sigma_{nr}^{2}.$ If $\underset{n\rightarrow\infty}{\lim\inf}\frac{s_{n}^{2}%
}{n}>0,$ then%
\begin{equation}
\mathcal{P}\left(  \underset{n\rightarrow\infty}{\lim\sup}\frac{S_{n}}%
{\sqrt{2s_{n}^{2}\log\log s_{n}^{2}}}\geq1\right)  =1, \label{LIL1}%
\end{equation}
and%
\begin{equation}
\mathcal{P}\left(  \underset{n\rightarrow\infty}{\lim\inf}\frac{S_{n}}%
{\sqrt{2s_{n}^{2}\log\log s_{n}^{2}}}\leq-1\right)  =1. \label{LIL2}%
\end{equation}

\end{lemma}

\begin{proof}
This is a direct consequence of Corollary $1$ of Theorem $2$ in
\cite{Tomkins74}.
\end{proof}

The following theorem shows the infinitely many oscillations of the paths of
coupled harmonic oscillators (\ref{harmonic oscillator}), which extends the
Theorem $4.1$ in \cite{Mao07} (Section $8.4$) that refers to the paths of
simple harmonic oscillators (i.e., those defined by (\ref{harmonic oscillator}%
) with $d=1$).

\begin{theorem}
\label{Teor1}Consider the coupled harmonic oscillator
(\ref{harmonic oscillator}). Then, almost surely, each component of the
solution $\mathbf{x}(t)$ has infinitely many zeros on $[t_{0}$ $\infty)$ for
every $t_{0}\geq0.$
\end{theorem}

\begin{proof}
Let us start considering the first component $x^{1}$ of the solution of
(\ref{harmonic oscillator}). By the spectral theorem for the real symmetric
matrix $\Lambda$ we have the factorization%
\[
\Lambda=Pdiag[\lambda_{1},\ldots,\lambda_{d}]P^{\intercal},
\]
where $\lambda_{1},\ldots,\lambda_{d}$ are the eigenvalues of $\Lambda$, and
$P$ is a real orthogonal matrix with entries $[P_{k,j}]$ for $k,j=1,\ldots,d$.
Then, for $f(\Lambda)=\sin(\Lambda)$ and for $f(\Lambda)=\cos(\Lambda)$, we
have (see, e.g., \cite{Higham08})%
\[
f(\Lambda)=Pdiag[f(\lambda_{1}),\ldots,f(\lambda_{d})]P^{\intercal}.
\]
Since the solution of (\ref{harmonic oscillator}) satisfies (see, e.g.,
\cite{Mao07})%
\begin{align*}
\left[
\begin{array}
[c]{c}%
x\left(  t\right) \\
y\left(  t\right)
\end{array}
\right]   &  =\left[
\begin{array}
[c]{cc}%
\text{cos}(\Lambda(t-t_{0})) & \Lambda^{-1}\text{sin}(\Lambda(t-t_{0}))\\
-\Lambda\text{sin}(\Lambda(t-t_{0})) & \text{cos}(\Lambda(t-t_{0}))
\end{array}
\right]  \left[
\begin{array}
[c]{c}%
x_{0}\\
y_{0}%
\end{array}
\right] \\
&  +%
{\displaystyle\int\limits_{t_{0}}^{t}}
\begin{bmatrix}
\Lambda^{-1}\sin(\Lambda\left(  t-s\right)  )\\
\cos(\Lambda\left(  t-s\right)  )
\end{bmatrix}
\Pi d\mathbf{w}_{s},
\end{align*}
then%
\begin{equation}
x^{1}(t)=D(t)+V(t), \label{s1}%
\end{equation}
where%
\[
D(t)=\sum_{k=1}^{d}\left(  P_{1k}\cos(\lambda_{k}(t-t_{0}))\left\langle
P_{k},x_{0}\right\rangle +P_{1k}\lambda_{k}^{-1}\sin(\lambda_{k}%
(t-t_{0}))\left\langle P_{k},y_{0}\right\rangle \right)  ,
\]
and%
\begin{equation}
V(t)=%
{\textstyle\sum\limits_{l=1}^{m}}
\left(
{\textstyle\int\limits_{t_{0}}^{t}}
\left(
{\textstyle\sum\limits_{k=1}^{d}}
c_{k}^{l}\sin\left(  \lambda_{k}\left(  t-s\right)  \right)  \right)
d\mathbf{w}_{s}^{l}\right)  , \label{V_(t)}%
\end{equation}
being $c_{k}^{l}=P_{1k}\lambda_{k}^{-1}\left\langle P_{k},\Pi_{l}\right\rangle
$, and $P_{k},$ $\Pi_{l}$ the column vectors of $P$ and $\Pi$, respectively.

Without loss of generality, let us assume that $\lambda_{k}>0$ and
$\lambda_{k}\neq$ $\lambda_{r}$ for all $k\neq r$ with $k,r=1,\ldots,d$.
Indeed, when there are only $d^{\ast}<d$ different values $\lambda_{j}^{\ast}$
of $\left\vert \lambda_{k}\right\vert $, $k=1,\ldots,d$ and $j=1,\ldots
,d^{\ast}$, the expression (\ref{V_(t)}) can be rewritten as%
\[
V(t)=%
{\textstyle\sum\limits_{l=1}^{m}}
\left(
{\textstyle\int\limits_{t_{0}}^{t}}
\left(
{\textstyle\sum\limits_{j=1}^{d^{\ast}}}
e_{j}^{l}\sin\left(  \lambda_{j}^{\ast}\left(  t-s\right)  \right)  \right)
d\mathbf{w}_{s}^{l}\right)  ,
\]
where $e_{j}^{l}=%
{\displaystyle\sum\limits_{k=1}^{d}}
c_{k}^{l}\delta_{\lambda_{j}^{\ast}}^{\left\vert \lambda_{k}\right\vert
}(1_{\lambda_{k}>0}-1_{\lambda_{k}<0})$, and $\delta$ is the Kronecker delta.
For this expression of $V(t)$ the analysis below would be the same as that for
(\ref{V_(t)}) with the above mentioned assumptions on $\lambda_{k}$.

Consider an arbitrary $\Delta>0$ and the time instants $t_{n}=t_{0}+n\Delta$,
with $n=1,2,\ldots$. In addition, for all $n$, define%
\begin{equation}
S_{n}:=V(t_{n})=%
{\textstyle\sum\limits_{r=1}^{n}}
V_{nr}, \label{sum}%
\end{equation}
where $V(t_{n})$ is defined in (\ref{V_(t)}), and
\[
V_{nr}=%
{\textstyle\sum\limits_{l=1}^{m}}
\left(
{\textstyle\int\limits_{t_{0}+(r-1)\Delta}^{t_{0}+r\Delta}}
\left(
{\textstyle\sum\limits_{k=1}^{d}}
c_{k}^{l}\sin\left(  \lambda_{k}(t_{n}-s)\right)  \right)  d\mathbf{w}_{s}%
^{l}\right)  ,
\]
for all $n,r=1,2,\ldots$. Because the independence of $\mathbf{w}_{s}%
^{1},...,\mathbf{w}_{s}^{m}$ and the independence of the increments of
$\mathbf{w}_{s}^{l}$ on disjoint intervals, $\{V_{nr}\}_{r\geq1}$ defines a
double sequence of $i.i.d.$ Gaussian random variables with zero mean and
variance%
\begin{align}
\sigma_{nr}^{2}  &  =\mathbb{E}(V_{nr}^{2})\nonumber\\
&  =%
{\textstyle\sum\limits_{l=1}^{m}}
\text{ }%
{\textstyle\int\limits_{t_{0}+(r-1)\Delta}^{t_{0}+r\Delta}}
\left(
{\textstyle\sum\limits_{k=1}^{d}}
c_{k}^{l}\sin\left(  \lambda_{k}\left(  t_{n}-s\right)  \right)  \right)
^{2}ds. \label{Teor1.1}%
\end{align}
In this way, (\ref{sum}) can be written as
\[
S_{n}=%
{\textstyle\sum\limits_{r=1}^{n}}
\sigma_{nr}\eta_{r},
\]
where $\eta_{1},...,\eta_{n}$ are $i.i.d.$ $\mathcal{N}(0,1)$ random
variables. Thus, the variance $s_{n}^{2}$ of $S_{n}$ satisfies%
\[
s_{n}^{2}=%
{\textstyle\sum\limits_{r=1}^{n}}
\sigma_{nr}^{2}.
\]
The expression (\ref{Teor1.1}) and the identity $sin(\theta)=\left(
\exp(i\theta)-\exp(-i\theta)\right)  /(2i)$ (where $i=\sqrt{-1}$) imply that%
\begin{align*}
{\tiny s}_{n}^{2}  &  {\tiny =}{\tiny -}\frac{1}{2}%
{\textstyle\sum\limits_{l=1}^{m}}
{\textstyle\sum\limits_{k,j=1}^{d}}
{\tiny c}_{k}^{l}{\tiny c}_{j}^{l}\operatorname{Re}\left\{  \exp(\text{$i$%
}(\lambda_{j}+\lambda_{k})t_{n})%
{\textstyle\int\limits_{t_{0}}^{t_{n}}}
\exp(-\text{$i$}(\lambda_{j}+\lambda_{k})s)ds\right. \\
&  \left.  -\exp(\text{$i$}(\lambda_{j}-\lambda_{k})t_{n})%
{\textstyle\int\limits_{t_{0}}^{t_{n}}}
\exp(-\text{$i$}(\lambda_{j}-\lambda_{k})s)ds\right\}  {\tiny ,}%
\end{align*}
where $\operatorname{Re}$ denotes the real part of a complex number. Since%
\[%
{\textstyle\int\limits_{t_{0}}^{t_{n}}}
\exp(-\text{$i$}\theta s)ds=\left\{
\begin{array}
[c]{cc}%
n\Delta & \text{if }\theta=0\text{ mod }2\pi\\
\left(  \exp(-\text{$i$}\theta t_{0})-\exp(-\text{$i$}\theta t_{n})\right)
/(\text{$i$}\theta) & \text{otherwise}%
\end{array}
\right.  ,
\]
we have%
\[
s_{n}^{2}=\frac{1}{2}%
{\textstyle\sum\limits_{l=1}^{m}}
{\textstyle\sum\limits_{k=1}^{d}}
\left(  c_{k}^{l}\right)  ^{2}n\Delta+C_{n},
\]
where $C_{n}$ is uniformly bounded for all $n$. Thus,%
\[
\underset{n\rightarrow\infty}{\lim}\frac{s_{n}^{2}}{n}=\frac{1}{2}%
{\textstyle\sum\limits_{l=1}^{m}}
{\textstyle\sum\limits_{k=1}^{d}}
\left(  c_{k}^{l}\right)  ^{2}\Delta>0.
\]
In addition, since
\begin{align*}
\sigma_{nr}^{2}  &  \leq%
{\textstyle\sum\limits_{l=1}^{m}}
{\textstyle\int\limits_{t_{0}+(r-1)\Delta}^{t_{0}+r\Delta}}
\left(
{\textstyle\sum\limits_{k=1}^{d}}
\left\vert c_{k}^{l}\right\vert \right)  ^{2}ds\\
&  \leq\Delta d%
{\textstyle\sum\limits_{l=1}^{m}}
{\textstyle\sum\limits_{k=1}^{d}}
\left\vert c_{k}^{l}\right\vert ^{2},
\end{align*}
for all $n$ and $r$, the Law of the Iterated Logarithms of Lemma \ref{Lemma1}
holds for $S_{n}$. Thus, for $0<\varepsilon<1$, (\ref{LIL1}) implies that%
\[
S_{n}>\left(  1-\varepsilon\right)  \sqrt{2s_{n}^{2}\left(  \log\log s_{n}%
^{2}\right)  }\text{ \ for infinitely many values of }n\text{ (almost
surely).}%
\]
In addition, since \ \ \
\[
\left\vert D(t_{n})\right\vert \leq\left\vert P\right\vert ^{2}(\left\vert
x_{0}\right\vert +\left\vert y_{0}\right\vert \underset{k}{\max}\left\{
\lambda_{k}^{-1}\right\}  )
\]
for all $n$, for the fist component (\ref{s1}) of the solution of
(\ref{harmonic oscillator}) we have that
\[
x^{1}(t_{n})>0\text{ infinitely often as }n\rightarrow\infty\text{ (almost
surely).}%
\]
Similarly, (\ref{LIL2}) implies that
\[
S_{n}<\left(  -1+\varepsilon\right)  \sqrt{2s_{n}^{2}\left(  \log\log
s_{n}^{2}\right)  }\text{ \ for infinitely many values of }n\text{ (almost
surely)}%
\]
for $0<\varepsilon<1$, and so%
\[
x^{1}(t_{n})<0\text{ infinitely often as }n\rightarrow\infty\text{ (almost
surely).}%
\]
Thus, since the sample path of the solution to (\ref{harmonic oscillator}) is
continuous, $x^{1}(t)$ must have, almost surely, infinitely many zeros on
$[t_{0}$ $\infty)$. For the remainder of the components of the solution of
(\ref{harmonic oscillator}) we can proceed in a similar manner. This concludes
the proof. \
\end{proof}

\section{The infinitely many zeros of coupled nonlinear oscillators}

Let us consider the coupled nonlinear oscillator defined by the $2d$%
-dimensional ($d>1$) SDE with additive noise%

\begin{equation}%
\begin{array}
[c]{l}%
dx(t)=y(t)dt,\\
dy(t)=-f(x(t),y(t))dt+\Pi d\widetilde{\mathbf{w}}_{t},
\end{array}
\label{nonlinear oscillator}%
\end{equation}
where $\Pi\in%
\mathbb{R}
^{d\times m}$ is a matrix, $\widetilde{\mathbf{w}}_{t}$ is a $m$-dimensional
standard Wiener process on a filtered complete probability space different of
that of the equation (\ref{harmonic oscillator}), and $f:\mathbb{R}^{d}%
\times\mathbb{R}^{d}\rightarrow\mathbb{R}^{d}$ is a smooth function satisfying
the linear growth condition%
\begin{equation}
\left\vert f(x,y)\right\vert \leq K_{1}(1+\left\vert x\right\vert +\left\vert
y\right\vert ), \label{linear growth}%
\end{equation}
for some positive constant $K_{1}$.

For analysis of the oscillatory behavior of (\ref{nonlinear oscillator}), next
Lemma will be useful.

\begin{lemma}
\label{Lemma2}Let $\left(  x(t),y(t)\right)  ^{\top}\in%
\mathbb{R}
^{2d}$ be the unique solution of the harmonic oscillator equation
(\ref{harmonic oscillator}) on $[0,T]$ for any $T>0$. Suppose that $\Phi
_{t}:=\phi(x(t),y(t)):%
\mathbb{R}
\rightarrow%
\mathbb{R}
^{m}$ is a function satisfying the linear growth condition
\begin{equation}
\left\vert \phi(x(t),y(t))\right\vert \leq K(1+\left\vert x(t)\right\vert
+\left\vert y(t)\right\vert ). \label{phi linear growth}%
\end{equation}
Then, there is a probabilistic measure $\widetilde{\mathbb{P}}$ on $\left(
\Omega,\mathfrak{F}\right)  $ absolutely continuous with respect to
$\mathbb{P}$ and an $m$-dimensional standard Wiener process $\widetilde
{\mathbf{w}}_{t}$ on $\left(  \Omega,\mathfrak{F,}\left(  \mathfrak{F}%
_{t}\right)  _{t\geq t_{0}},\widetilde{\mathbb{P}}\right)  $ such that
$\left(  x(t),y(t)\right)  ^{\intercal}$ is also the unique solution of the
nonlinear equation
\begin{equation}%
\begin{array}
[c]{l}%
dx(t)=y(t)dt,\\
dy(t)=\left(  -\Lambda^{2}x(t)+\Pi\Phi_{t}\right)  dt+\Pi d\widetilde
{\mathbf{w}}_{t},
\end{array}
\label{nonlinear oscillator 2}%
\end{equation}
on $[0,T]$.
\end{lemma}

\begin{proof}
Let $\mathbf{x}_{t}=\left(  x(t),y(t)\right)  ^{\intercal}$ be the solution of
the equation (\ref{harmonic oscillator}) on $[0,T]$. From the condition
(\ref{phi linear growth}) it follows that%
\[
\left\vert \Phi_{t}\right\vert ^{2}\leq C\left(  1+\left\vert \mathbf{x}%
_{t}\right\vert ^{2}\right)  ,
\]
where $C=3K^{2}$.

Since $\mathbf{x}_{t}$ is the solution of the linear SDE with additive noise
(\ref{harmonic oscillator}), $\mathbf{x}_{t}\sim\mathcal{N}_{2d}(\mu
_{t},\Sigma_{t})$ for all $t\in\lbrack0,T]$, where the mean $\mu_{t}$ and the
variance $\Sigma_{t}$ of $\mathbf{x}_{t}$ are continuous functions on $[0,T]$
(see, e.g., \cite{Arnold 1974}). Here, $\mathcal{N}_{2d}$ denotes $2d-$variate
normal distribution. The random vector $\mathbf{x}_{t}$ can be written as
$\mathbf{x}_{t}=\mu_{t}+\Sigma_{t}^{1/2}Z_{t}$, where $\Sigma_{t}^{1/2}$ is
the symmetric square root of $\Sigma_{t}$, and $Z_{t}\sim\mathcal{N}%
_{2d}(\mathbf{0},\mathbf{I})$. Therefore,
\begin{align*}
\mathbb{E}\left(  \exp\left(  \left\vert \Phi_{t}\right\vert ^{2}\right)
\right)   &  \leq\exp\left(  C\right)  \mathbb{E}\left(  \exp\left(
C\left\vert \mu_{t}+\Sigma_{t}^{1/2}Z_{t}\right\vert ^{2}\right)  \right) \\
&  \leq\exp\left(  C+2C\left\vert \mu_{t}\right\vert ^{2}\right)
\mathbb{E}\left(  \exp\left(  2C\left\vert \Sigma_{t}^{1/2}\right\vert
^{2}\left\vert Z_{t}\right\vert ^{2}\right)  \right)  .
\end{align*}

Since $\left\vert Z_{t}\right\vert ^{2}$ is a random variable that has
chi-squared distribution with $2d$ degrees of freedom, $E\left(  \exp\left(
\alpha\left\vert Z_{t}\right\vert ^{2}\right)  \right)  \leq1/\left(
1-2\alpha\right)  ^{d}$ for $\alpha<1/2$ (\cite{Johnson94}, pp. 420).
Therefore, for all $a<1/\left(  8C\max_{t\in\lbrack0,T]}\left\vert \Sigma
_{t}\right\vert \right)  $, it holds that%
\begin{align*}
\mathbb{E}\left(  \exp\left(  a\left\vert \Phi_{t}\right\vert ^{2}\right)
\right)   &  \leq\mathbb{E}\left(  \exp\left(  \frac{1}{4}\left\vert
Z_{t}\right\vert ^{2}\right)  \right)  \\
&  \leq2^{d}D,
\end{align*}
where $D=$ $\exp\left(  aC+2aC\max_{t\in\lbrack0,T]}\left\vert \mu
_{t}\right\vert ^{2}\right)  $. The proof is then completed as a direct
consequence of the Cameron-Martin-Girsanov theorem (see, e.g., \cite{Mao07},
pp. 274).
\end{proof}

Next theorem provides conditions that guarantee a link between the solutions
of the harmonic and nonlinear oscillator equations.

\begin{theorem}
\label{Teor2}Let $\left(  x(t),y(t)\right)  ^{\intercal}\in%
\mathbb{R}
^{2d}$ be the unique solution of the harmonic oscillator equation
(\ref{harmonic oscillator}) on $[0,T]$ for $T>0$. Let $\Phi_{t}:=\phi
(x(t),y(t)):%
\mathbb{R}
\rightarrow%
\mathbb{R}
^{m}$ be a function such that%
\begin{equation}
\Pi\Phi_{t}=\Lambda^{2}x(t)-f(x(t),y(t)), \label{phi def}%
\end{equation}
where the function $f$ satisfies the linear growth condition
(\ref{linear growth}). Then, there is a probabilistic measure $\widetilde
{\mathbb{P}}$ on $\left(  \Omega,\mathfrak{F}\right)  $ absolutely continuous
with respect to $\mathbb{P}$ and an $m$-dimensional standard Wiener process
$\widetilde{\mathbf{w}}_{t}$ on $\left(  \Omega,\mathfrak{F,}\left(
\mathfrak{F}_{t}\right)  _{t\geq t_{0}},\widetilde{\mathbb{P}}\right)  $ such
that $\left(  x(t),y(t)\right)  ^{\intercal}$ is also the unique solution of
the nonlinear oscillator equation (\ref{nonlinear oscillator}) on $[0,T]$.
\end{theorem}

\begin{proof}
Since $\Phi_{t}$ solves the equation (\ref{phi def}), $\Phi_{t}=\Pi^{-}\left(
\Lambda^{2}x(t)-f(x(t),y(t))\right)  $, where the matrix $\Pi^{-}$ is a
generalized inverse of $\Pi$. This and condition (\ref{linear growth}) imply
that $\Phi_{t}$ satisfies the linear growth condition (\ref{phi linear growth}%
). Then, the assumptions of Lemma \ref{Lemma2} are fulfilled, which completes
the proof.
\end{proof}

Notice that the assumptions of Theorem \ref{Teor2} are directly satisfied in
the case, for instance, that $\Pi$ in (\ref{harmonic oscillator}) is a
nonsingular $d\times d$ matrix.

Next theorem deals with the infinite oscillations of the paths of the coupled
nonlinear oscillator (\ref{nonlinear oscillator}).

\begin{theorem}
\label{Teor3} Each component of the solution of the coupled nonlinear
oscillator (\ref{nonlinear oscillator}) has infinitely many zeros on $[t_{0}$
$\infty)$ for every $t_{0}\geq0$ almost surely.
\end{theorem}

\begin{proof}
Theorem \ref{Teor2} states that, for properties holding almost surely, the
analysis of the nonlinear oscillator (\ref{nonlinear oscillator}) with growth
condition (\ref{linear growth}) reduces to that of the harmonic oscillator
(\ref{harmonic oscillator}). In this way, since by Theorem \ref{Teor1} each
component of the harmonic oscillator (\ref{harmonic oscillator}) has
infinitely many zeros on $[t_{0}$ $\infty)$, each component of the nonlinear
oscillator (\ref{harmonic oscillator}) will also has infinitely many zeros on
$[t_{0}$ $\infty)$ for every $t_{0}\geq0$.
\end{proof}

As example of equation (\ref{nonlinear oscillator}), with condition
(\ref{linear growth}) being satisfied, we can mention the equation of various
type of coupled pendulums driven by random forces, as those of
\cite{Meirovitch86}: e.g., 1) a \textit{double pendulum} (a pendulum with
another pendulum attached to its end); and 2) a pair of identical pendulums
connected by a weak spring. For the last one, we have the equation%
\begin{align*}
dx_{1}\left(  t\right)   &  =y_{1}\left(  t\right)  dt,\\
dy_{1}\left(  t\right)   &  =-\alpha\sin\left(  x_{1}\left(  t\right)
\right)  -\beta\left(  \sin\left(  x_{1}\left(  t\right)  \right)
-\sin\left(  x_{2}\left(  t\right)  \right)  \right)  \cos\left(  x_{1}\left(
t\right)  \right)  dt+\sigma_{1}d\mathbf{w}_{t}^{1},\\
dx_{2}\left(  t\right)   &  =y_{2}\left(  t\right)  dt,\\
dy_{2}\left(  t\right)   &  =-\alpha\sin\left(  x_{2}\left(  t\right)
\right)  +\beta\left(  \sin\left(  x_{1}\left(  t\right)  \right)
-\sin\left(  x_{2}\left(  t\right)  \right)  \right)  \cos\left(  x_{2}\left(
t\right)  \right)  dt+\sigma_{2}d\mathbf{w}_{t}^{2},
\end{align*}
with $\alpha,$ $\beta,$ $\sigma_{1},$ $\sigma_{2}$ $\in%
\mathbb{R}
^{+}$.

Clearly for this equation there is a function $\Phi_{t}$ satisfying
(\ref{phi def}). Thus, by Theorem \ref{Teor3}, each component of this
nonlinear equation has infinitely many zeros on $[t_{0}$ $\infty)$ for every
$t_{0}\geq0$ almost surely.

\section{Simplicity of the zeros}

This section deals with the simplicity of the zeros of coupled harmonic and
coupled nonlinear stochastic oscillators considered in previous sections. That
is, we will prove that the component $y^{i}$ of the oscillators does not
vanish at the same time that the component $x^{i}$ does.

\begin{theorem}
\label{Teor4}The infinite many zeros of the coupled harmonic oscillator
(\ref{harmonic oscillator}) are simple.
\end{theorem}

\begin{proof}
First, note that Theorem \ref{Teor2} implies that there is a probabilistic
measure $\widetilde{\mathbb{P}}$ on $\left(  \Omega,\mathfrak{F}\right)  $
absolutely continuous with respect to $\mathbb{P}$ and an $m$-dimensional
standard Wiener process $\widetilde{\mathbf{w}}_{t}$ on $\left(
\Omega,\mathfrak{F,}\left(  \mathfrak{F}_{t}\right)  _{t\geq t_{0}}%
,\widetilde{\mathbb{P}}\right)  $ such that the solution $\left(
x(t),y(t)\right)  $ of the coupled harmonic oscillators
(\ref{harmonic oscillator}) is also the unique solution of the discoupled
oscillator%
\begin{equation}
d%
\begin{bmatrix}
x(t)\\
y(t)
\end{bmatrix}
=%
\begin{bmatrix}
\mathbf{0} & \mathbf{I}\\
-\mathbf{D}^{2} & \mathbf{0}%
\end{bmatrix}%
\begin{bmatrix}
x(t)\\
y(t)
\end{bmatrix}
dt+%
\begin{bmatrix}
\mathbf{0}\\
\Pi
\end{bmatrix}
d\widetilde{\mathbf{w}}_{t}, \label{discoupled harmonic oscillator}%
\end{equation}
on $[t_{0},T]$, with initial condition $\mathbf{x}(t_{0})=(x_{0},y_{0})^{\top
}$, being $\mathbf{D}\in%
\mathbb{R}
^{d\times d}$ a diagonal matrix and $\Pi$ defined as in
(\ref{harmonic oscillator}).

When $\mathbf{w}_{t}$ in (\ref{harmonic oscillator}) is a one-dimensional
standard Wiener process, the simplicity of the zeros of each component $x^{i}$
of (\ref{harmonic oscillator}) is a straightforward consequence of the
mentioned in the previous paragraph and Theorem 4.1, pp. 280, in \cite{Mao07}
for the simplicity of the zeros of the simple harmonic oscillator with
one-dimensional Wiener process.

When $m>1$, by following similar ideas of the proof of Theorem 3.4, pp. 277,
in \cite{Mao07}, the simplicity of the zeros of the simple harmonic oscillator%
\begin{equation}
d%
\begin{bmatrix}
x(t)\\
y(t)
\end{bmatrix}
=%
\begin{bmatrix}
0 & 1\\
-\alpha^{2} & 0
\end{bmatrix}%
\begin{bmatrix}
x(t)\\
y(t)
\end{bmatrix}
dt+%
\begin{bmatrix}
0\\%
{\textstyle\sum\limits_{j=1}^{m}}
\sigma_{j}d\widetilde{\mathbf{w}}_{t}^{j}%
\end{bmatrix}
, \label{Linear_Oscillator_dim2}%
\end{equation}
with $\alpha\in%
\mathbb{R}
,$ can be proved as follows.

We will first ensure the existence of a function $V(x,y)>0$ such
that%
\begin{equation}
\lim_{\left\vert x\right\vert +\left\vert y\right\vert \rightarrow
0}V(x,y)=\infty\text{ \ \ \ and \ \ \ }\mathbb{E}\left(  V(x(t),y(t))\right)
=\mathbb{E}\left(  V(x_{0},y_{0})\right)  \text{ for all }t\geq t_{0}\text{.}
\label{conditions for V}%
\end{equation}
From the It\^{o}-formula%
\begin{equation}
V(x(t),y(t))=V(x_{0},y_{0})+%
{\textstyle\int\limits_{t_{0}}^{t}}
LV(x_{s},y_{s})ds+%
{\textstyle\sum\limits_{j=1}^{m}}
{\textstyle\int\limits_{t_{0}}^{t}}
\frac{\partial V}{\partial y}(x_{s},y_{s})\sigma_{j}d\widetilde{\mathbf{w}%
}_{s}^{j}, \label{Ito_for_V}%
\end{equation}
with the operator%
\[
L=y\frac{\partial}{\partial x}-\alpha^{2}x\frac{\partial}{\partial y}+\frac
{1}{2}\left(
{\textstyle\sum\limits_{j=1}^{m}}
\left(  \sigma_{j}\right)  ^{2}\right)  \frac{\partial^{2}}{\partial^{2}y},
\]
it is easy to check that if $LV(x,y)=0$ then $\mathbb{E}\left(
V(x(t),y(t))\right)  =\mathbb{E}\left(  V(x_{0},y_{0})\right)  $. In addition,
note that the operator $L$ defines the Forward-Kolmogorov (Fokker-Planck)
equation%
\begin{equation}
Lp(t,x,y)=\frac{\partial p(t,x,y)}{\partial t}, \label{FPE}%
\end{equation}
for the transition probability function $p$ corresponding to the solution of
the linear SDE%
\[
d%
\begin{bmatrix}
x(t)\\
y(t)
\end{bmatrix}
=%
\begin{bmatrix}
0 & -1\\
\alpha^{2} & 0
\end{bmatrix}%
\begin{bmatrix}
x(t)\\
y(t)
\end{bmatrix}
dt+%
\begin{bmatrix}
0\\
\rho d\bar{W}_{t}%
\end{bmatrix}
,
\]
for some scalar Wiener process $\bar{W}_{t}$ with $\rho=%
{\textstyle\sum\limits_{j=1}^{m}}
\sigma_{j}^{2}$. Thus,%
\[
p(t,x,y)=\frac{1}{2\pi\left\vert \Sigma\right\vert ^{1/2}}\exp(-\frac{1}{2}%
\begin{bmatrix}
x & y
\end{bmatrix}
\Sigma^{-1}%
\begin{bmatrix}
x & y
\end{bmatrix}
^{\intercal}),
\]
with%
\[
\Sigma=%
\begin{bmatrix}
\frac{\rho\left(  2\alpha t-\sin(2\alpha t)\right)  }{4\alpha^{3}} &
-\frac{\rho\sin^{2}(\alpha t)}{2\alpha^{2}}\\
-\frac{\rho\sin^{2}(\alpha t)}{2\alpha^{2}} & \frac{\rho\left(  2\alpha
t+\sin(2\alpha t)\right)  }{4\alpha}%
\end{bmatrix}
.
\]
After some algebraic manipulation we obtain that%
\[
p(t,x,y)=\frac{\alpha^{2}}{\pi\rho^{2}\left(  (\alpha t)^{2}-\sin^{2}(\alpha
t)\right)  ^{1/2}}\exp(M(t,x,y)),
\]
with%
\[
M(t,x,y)=-\frac{2t\left(  \alpha y\right)  ^{2}-y^{2}\alpha\sin(2\alpha
t)+2x^{2}\alpha^{3}t+2xy\alpha+(x\alpha)^{2}\sin(2\alpha t)-2xy\alpha
\cos(2\alpha t)}{\rho^{2}(\cos(2\alpha t)+2(\alpha t)^{2}-1)}.
\]
From this and (\ref{FPE}), it is easy to check that%
\[
V(x,y)=%
{\textstyle\int\limits_{0}^{\infty}}
p(s,x,y)ds,
\]
satisfies the conditions (\ref{conditions for V}).

Now, let us prove that for the stopping time%
\[
\tau=\inf\left\{  t\geqslant t_{0}:x(t)=y(t)=0\right\}  ,
\]
$\widetilde{\mathbb{P}}\mathbb{(}\omega\in\Omega:\tau\leq T)=0$, for arbitrary
$T>t_{0}$. For this, let us define additional stopping times%
\[
\tau_{n}=\inf\left\{  t\geqslant t_{0}:x^{2}(t)+y^{2}(t)\leq\frac{1}%
{n}\right\}  ,
\]
for $n=1,2,...$. Using condition (\ref{conditions for V}) and that
$(x(\tau_{n}\wedge T),y(\tau_{n}\wedge T))=(x(\tau_{n}),y(\tau_{n}))$ for all
$\omega\in\Gamma=\left\{  \omega\in\Omega:\text{ }\tau\leq T\right\}
\subset\left\{  \omega\in\Omega:\tau_{n}\leq T\right\}  $, we have%
\begin{align*}
\mathbb{E}\left(  V(x_{0},y_{0})\right)   &  =\mathbb{E}\left(  V(x(\tau
_{n}\wedge T),y(\tau_{n}\wedge T))\right) \\
&  \geqslant%
{\textstyle\int\limits_{\Gamma}}
V(x(\tau_{n}\wedge T),y(\tau_{n}\wedge T))d\widetilde{\mathbb{P}}\\
&  \geqslant%
{\textstyle\int\limits_{\Gamma}}
V(x(\tau_{n}),y(\tau_{n}))d\widetilde{\mathbb{P}}.
\end{align*}
In this way,%
\[
\mathbb{E}\left(  V(x_{0},y_{0})\right)  \geq%
{\textstyle\int\limits_{\Gamma}}
\lim_{n\longrightarrow\infty}V(x(\tau_{n}),y(\tau_{n}))d\widetilde{\mathbb{P}%
}.
\]
From (\ref{conditions for V}) and since $\left\vert x(\tau_{n})\right\vert
+\left\vert y(\tau_{n})\right\vert \leq\sqrt{2\left(  \left\vert x(\tau
_{n})\right\vert ^{2}+\left\vert y(\tau_{n})\right\vert ^{2}\right)  }%
\leq\sqrt{\frac{2}{n}}$, we have $\lim_{n\longrightarrow\infty}V(x(\tau
_{n}),y(\tau_{n}))=\infty$. Thus, since $\mathbb{E}\left(  V(x_{0}%
,y_{0})\right)  <\infty$, necessarily $\widetilde{\mathbb{P}}(\Gamma)=0$
holds. That is, all the zeros of (\ref{Linear_Oscillator_dim2}) are simple.
From this and taking into account that the coupled (\ref{harmonic oscillator})
and discoupled (\ref{discoupled harmonic oscillator}) harmonic oscillators
with $m\geq1$ have the same solution, almost surely, the proof is completed.
\end{proof}

For the zeros of coupled nonlinear oscillators we have the following result.

\begin{theorem}
The infinite many zeros of the coupled nonlinear oscillator
(\ref{nonlinear oscillator}) are simple.
\end{theorem}

\begin{proof}
It is a straightforward consequence of Theorem \ref{Teor2} and Theorem
\ref{Teor4} above.
\end{proof}

\section{The infinitely many zeros of the Local Linearized integrators for
coupled harmonic oscillators}

Let $(t)_{h}=\{t_{n}=t_{0}+nh:n=0,1,\ldots\},$ $h>0,$ be a partition of the
time interval $[t_{0},\infty)$. The Locally Linearized integrator for the
equation (\ref{harmonic oscillator}) is defined by the recursive expression
\cite{delaCruz16}%
\begin{equation}
\mathbf{x}_{n+1}=\mathbf{x}_{n}+\mathbf{u}_{n}+\mathbf{z}_{n+1},
\label{LL integrator oscillator}%
\end{equation}
for $n=0,1,\ldots$, with initial condition $\mathbf{x}_{0}=\mathbf{x}(t_{0})$,
where $\mathbf{x}_{n}=(x_{n},y_{n})^{\top},$ $x_{n},y_{n}\in%
\mathbb{R}
^{d}$, $\mathbf{u}_{n}=\mathbf{L}e^{\mathbf{C}_{n}h}\mathbf{r}$, and
$\mathbf{z}_{n+1}=\mathbf{Q}\Delta\mathbf{w}_{n}$. Here%
\[
\mathbf{C}_{n}=\left[
\begin{array}
[c]{ccc}%
\mathbf{0} & \mathbf{I} & y_{n}\\
-\Lambda^{2} & \mathbf{0} & -\Lambda^{2}x_{n}\\
\mathbf{0}_{1\times d} & \mathbf{0}_{1\times d} & 0
\end{array}
\right]  \in\mathbb{R}^{(2d+1)\times(2d+1)},
\]
$\mathbf{L}=\left[  \mathbf{I}_{2d}\text{ \ }\mathbf{0}_{2d\times1}\right]  $,
$\mathbf{r}=\left[  \mathbf{0}_{1\times2d}\text{ }1\right]  ^{\intercal}$,
$\Delta\mathbf{w}_{n}=\mathbf{w}_{t_{n+1}}-\mathbf{w}_{t_{n}}$, and
$\mathbf{Q}=[%
\begin{array}
[c]{c}%
Q^{1}\\
Q^{2}%
\end{array}
]$ is a $2d\times m$ matrix with $Q^{1}$, $Q^{2}\in\mathbf{%
\mathbb{R}
}^{d\times m}$. The Locally Linearized integrator $\mathbf{x}_{n+1}$
converges, strongly with order $1$, to the solution $\mathbf{x}(t_{n+1})$ of
(\ref{harmonic oscillator}) at $t_{n+1}$ as $h$ goes to zero (\cite{Jimenez
12}, \cite{delaCruz16}).

Next theorem deals with the reproduction of the oscillatory behavior of
coupled harmonic oscillators by the discrete dynamical system defined by the
Locally Linearized integrator.

\begin{theorem}
\label{TeoremaLL}Let $\lambda_{1},\ldots,\lambda_{d}$ be the eigenvalues of
$\Lambda$, and $\left\vert \lambda\right\vert _{\max}=\max\limits_{k}\left(
\left\vert \lambda_{k}\right\vert \right)  $. For the coupled harmonic
oscillator (\ref{harmonic oscillator}), each component of the Locally
Linearized integrator switches signs infinitely many times as $n\rightarrow
\infty$, almost surely, for any integration stepsize $h<\pi/\left\vert
\lambda\right\vert _{\max}$.
\end{theorem}

\begin{proof}
Lemma 3.2 in \cite{delaCruz16} states that the Locally Linearized integrator
(\ref{LL integrator oscillator}) can be written as%
\begin{equation}
\mathbf{x}_{n+1}=\mathbf{M}^{n+1}\mathbf{x}_{0}+%
{\displaystyle\sum\limits_{r=0}^{n}}
\mathbf{M}^{r}\mathbf{Q}\Delta\mathbf{w}_{n-r},\nonumber
\end{equation}
where%
\[
\mathbf{M}^{r}=%
\begin{bmatrix}
\text{cos}\left(  r\Lambda h\right)   & \Lambda^{-1}\text{sin}\left(  r\Lambda
h\right)  \\
-\Lambda\text{sin}\left(  r\Lambda h\right)   & \text{cos}\left(  r\Lambda
h\right)
\end{bmatrix}
.
\]
Likewise in the proof of Theorem \ref{Teor1}, by using the Spectral Theorem,
the first component $x_{n+1}^{1}$ of $\mathbf{x}_{n+1}$ can be written%
\begin{equation}
x_{n+1}^{1}=D_{n+1}+S_{n},\label{Teor2.3}%
\end{equation}
where%
\[
D_{n+1}=\sum_{k=1}^{d}\left(  P_{1k}\cos((n+1)h\lambda_{k})\left\langle
P_{k},x_{0}\right\rangle +P_{1k}\lambda_{k}^{-1}\sin((n+1)h\lambda
_{k})\left\langle P_{k},y_{0}\right\rangle \right)  ,
\]
and%
\begin{equation}
S_{n}=%
{\displaystyle\sum\limits_{r=0}^{n}}
V_{nr},\label{Teor2.5}%
\end{equation}
being%
\begin{equation}
V_{nr}=%
{\displaystyle\sum\limits_{l=1}^{m}}
{\displaystyle\sum\limits_{j=1}^{d}}
\left(  e_{j}^{l}\cos(r\lambda_{j}h)+f_{j}^{l}{}\sin(r\lambda_{j}h)\right)
\Delta\mathbf{w}_{n-r}^{l},\label{Teor2.7}%
\end{equation}
with $e_{j}^{l}=P_{1j}\left\langle P_{j},Q_{l}^{1}\right\rangle $, $f_{j}%
^{l}=P_{1j}\lambda_{j}^{-1}\left\langle P_{j},Q_{l}^{2}\right\rangle $ and
$Q_{l}^{1},$ $Q_{l}^{2}$ the column vectors of $Q^{1}$ and $Q^{2}$, respectively.

Without loss of generality, let us assume that $\lambda_{k}>0$ and
$\lambda_{k}\neq$ $\lambda_{r}$ for all $k\neq r$ with $k,r=1,\ldots,d$.
Indeed, when there are only $d^{\ast}<d$ different values $\lambda_{j}^{\ast}$
of $\left\vert \lambda_{k}\right\vert $, $k=1,\ldots,d$ and $j=1,\ldots
,d^{\ast}$, the expression (\ref{Teor2.7}) can be rewritten as
\[
V_{nr}=%
{\displaystyle\sum\limits_{l=1}^{m}}
{\displaystyle\sum\limits_{j=1}^{d^{\ast}}}
\left(  E_{j}^{l}\cos(r\lambda_{j}^{\ast}h)+F_{j}^{l}{}\sin(r\lambda
_{j}h)\right)  \Delta\mathbf{w}_{n-r}^{l},
\]
where $E_{j}^{l}=%
{\displaystyle\sum\limits_{k=1}^{d}}
e_{k}^{l}\delta_{\lambda_{j}^{\ast}}^{\left\vert \lambda_{k}\right\vert }$,
$F_{j}^{l}=%
{\displaystyle\sum\limits_{k=1}^{d}}
f_{k}^{l}\delta_{\lambda_{j}^{\ast}}^{\left\vert \lambda_{k}\right\vert
}(1_{\lambda_{k}>0}-1_{\lambda_{k}<0})$, and $\delta$ is the Kronecker delta.
For this expression of $V_{nr}$ the analysis below would be the same as that
for (\ref{Teor2.7}) with the above mentioned assumptions on $\lambda_{k}$.

Since $\Delta\mathbf{w}_{n-r}^{l}$ are independent Gaussian random variables
with zero mean and variance $h$, $\{V_{nr}\}$ defines a sequence of $i.i.d.$
Gaussian random variable with zero mean and variance%
\[
\sigma_{nr}^{2}=h%
{\displaystyle\sum\limits_{l=1}^{m}}
\left(
{\displaystyle\sum\limits_{j=1}^{d}}
\left(  e_{j}^{l}\cos(r\lambda_{j}h)+f_{j}^{l}{}\sin(r\lambda_{j}h)\right)
\right)  ^{2}.
\]
Thus, (\ref{Teor2.5}) can be rewritten as%
\[
S_{n}=%
{\textstyle\sum\limits_{r=0}^{n}}
\sigma_{nr}\eta_{r},
\]
where $\eta_{0},...,\eta_{n}$ are $i.i.d.$ $\mathcal{N}(0,1)$ random
variables. Thus, the variance $s_{n}^{2}$ of $S_{n}$ satisfies%
\begin{equation}
s_{n}^{2}=%
{\displaystyle\sum\limits_{r=0}^{n}}
\sigma_{nr}^{2}.\label{Teor2.1}%
\end{equation}
Note that, for all $j=1,...,d$,%
\[%
{\displaystyle\sum\limits_{l=1}^{m}}
e_{j}^{l}\cos(r\lambda_{j}h)+%
{\displaystyle\sum\limits_{l=1}^{m}}
f_{j}^{l}{}\sin(r\lambda_{j}h)=c_{j}\cos\left(  r\lambda_{j}h-\alpha
_{j}\right)  ,
\]
where $c_{j}^{2}=\left(
{\displaystyle\sum\limits_{l=1}^{m}}
e_{j}^{l}\right)  ^{2}+\left(
{\displaystyle\sum\limits_{l=1}^{m}}
f_{j}^{l}{}\right)  ^{2}$, $\alpha_{j}=\arctan\left(
{\displaystyle\sum\limits_{l=1}^{m}}
f_{j}^{l}{}/%
{\displaystyle\sum\limits_{l=1}^{m}}
e_{j}^{l}\right)  $ for $%
{\displaystyle\sum\limits_{l=1}^{m}}
e_{j}^{l}\neq0$, and $\alpha_{j}=\frac{\pi}{2}$ for $%
{\displaystyle\sum\limits_{l=1}^{m}}
e_{j}^{l}=0.$ \newline From this and by using the identity $\cos
(\theta)=\left(  \exp(i\theta)+\exp(-i\theta)\right)  /2$, we obtain that%
\begin{align}
\sigma_{nr}^{2} &  =h\left(
{\displaystyle\sum\limits_{j=1}^{d}}
c_{j}\cos\left(  r\lambda_{j}h-\alpha_{j}\right)  \right)  ^{2}\nonumber\\
&  =\frac{h}{2}%
{\textstyle\sum\limits_{j,k=1}^{d}}
c_{k}c_{j}\operatorname{Re}\left\{  \exp(\text{$i$}r(\lambda_{k}+\lambda
_{j})h)\exp(-\text{$i$}(\alpha_{k}+\alpha_{j}))\right.  \label{Teor2.2}\\
&  +\left.  \exp(\text{$i$}r(\lambda_{k}-\lambda_{j})h)\exp(\text{$i$}%
(\alpha_{j}-\alpha_{k}))\right\}  ,\nonumber
\end{align}
where $\operatorname{Re}$ denotes the real part of a complex number. Under the
assumption $h<\pi/\left\vert \lambda\right\vert _{\max}$, it holds that
$\theta\neq0\operatorname{mod}2\pi$ for all $\theta=h\left(  \lambda
_{j}+\lambda_{k}\right)  $ with $1\leq j,k\leq d$, and for all $\theta
=h\left(  \lambda_{k}-\lambda_{j}\right)  $ with $1\leq j\neq k\leq d$.
Therefore, from (\ref{Teor2.2}) and the known expression of the partial sum of
the geometric series%
\[%
{\displaystyle\sum\limits_{r=0}^{n}}
\exp(\text{$i$}r\theta)=\left\{
\begin{array}
[c]{cc}%
n+1 & \text{if }\theta=0\text{ mod }2\pi\\
\frac{1-\exp(\text{$i$}\theta\left(  n+1\right)  )}{1-\exp(\text{$i$}\theta)}
& \text{otherwise}%
\end{array}
\right.  ,
\]
it is obtained that%
\[
s_{n}^{2}=n\frac{h}{2}%
{\textstyle\sum\limits_{k=1}^{d}}
c_{k}^{2}+C_{n},
\]
where $C_{n}$ is uniformly bound for all $n$. Thus, the assumption
$h<\pi/\left\vert \lambda\right\vert _{\max}$ implies that%
\[
\lim_{n\rightarrow\infty}\frac{s_{n}^{2}}{n}=\frac{h}{2}%
{\textstyle\sum\limits_{k=1}^{d}}
c_{k}^{2}>0.
\]
Since $\sigma_{nr}^{2}$ is bounded for all $n$ and $r$, the Law of the
Iterated Logarithms stated in Lemma \ref{Lemma1} holds for $S_{n}$. Thus, for
$0<\varepsilon<1$, (\ref{LIL1}) implies that%
\[
S_{n}>\left(  1-\varepsilon\right)  \sqrt{2s_{n}^{2}\left(  \log\log s_{n}%
^{2}\right)  }\text{ \ for infinitely many values of }n\text{ (almost
surely.).}%
\]
In addition, since%
\[
\left\vert D_{n+1}\right\vert \leq\left\vert P\right\vert ^{2}(\left\vert
x_{0}\right\vert +\left\vert y_{0}\right\vert \max\limits_{k}\left\{
\lambda_{k}^{-1}\right\}  ),
\]
for all $n$, the fist component (\ref{Teor2.3}) of the Locally Linearized
integrator (\ref{LL integrator oscillator}) satisfies
\[
x_{n+1}^{1}>0\text{ infinitely often as }n\rightarrow\infty\text{ (almost
surely).}%
\]
Similarly, (\ref{LIL2}) implies that
\[
S_{n}<\left(  -1+\varepsilon\right)  \sqrt{2s_{n}^{2}\left(  \log\log
s_{n}^{2}\right)  }\text{ for infinitely many values of }n\text{ (almost
surely),}%
\]
for $0<\varepsilon<1$, and so%
\[
x_{n+1}^{1}<0\text{ infinitely often as }n\rightarrow\infty\text{ (almost
surely).}%
\]
We can proceed similarly to prove that the other components of $\mathbf{x}%
_{n+1}$ also change sign infinitely often. This completes the proof.
\end{proof}

It was shown in \cite{delaCruz16} that, likewise the exact solution of the
simple harmonic oscillator (equation (\ref{harmonic oscillator}) with $d=1)$,
the path of the Local Linearized integrator (\ref{LL integrator oscillator})
switches signs infinitely many times as $n\rightarrow\infty$ almost surely for
any integration stepsize $h$. However, according to Theorem \ref{TeoremaLL},
in the case of the coupled oscillator (\ref{harmonic oscillator}), this
dynamics of the Local Linearized integrator (\ref{LL integrator oscillator})
is only guaranteed for stepsize $h<\pi/\max\left(  \left\vert \lambda
_{1}\right\vert ,\ldots,\left\vert \lambda_{d}\right\vert \right)  $, where
$\lambda_{1},\ldots,\lambda_{d}$ are the eigenvalues of $\Lambda$.

Theorem \ref{TeoremaLL} complements the results obtained in \cite{delaCruz16}
that demonstrate the capability of the discrete dynamical system defined by
the Local Linearized integrators for reproducing other essential continuous
dynamics of the coupled harmonic oscillators: the linear growth of energy
along the paths, and the symplectic structure of Hamiltonian oscillators.

Furthermore, since the exponential and trigonometric integrators considered in
\cite{Tocino07} and \cite{Cohen12} reduce to the expression
(\ref{LL integrator oscillator}) when they are applied to equation
(\ref{harmonic oscillator}), the Theorem \ref{TeoremaLL} can be also applied
for these integrators. In this way, these integrators with stepsize
$h<\pi/\max\left(  \left\vert \lambda_{1}\right\vert ,\ldots,\left\vert
\lambda_{d}\right\vert \right)  $ also switch signs infinitely many times as
$n\rightarrow\infty$ almost surely.

\section{Conclusion}

In this work, previous results concerning the infinitely many zeros of the
single harmonic oscillators driven by random forces were extended to the
general class of coupled harmonic oscillators. Furthermore, various classes of
coupled nonlinear oscillators having this oscillatory behavior were
identified. The ability of the discrete dynamical system defined by various
numerical integrators for reproducing this oscillatory property of the
continuous systems was also analyzed, which complements known results of these
integrators for the simple harmonic oscillators driven by random
forces.\newline

\textbf{Acknowledgements} \ The authors thank the financial support of a
FGV/EMAp project, Brazil, and Centro de Investigaci\'{o}n en Matem\'{a}ticas
(CIMAT), Mexico.

\end{document}